\documentclass[11pt,fleqn,a4paper]{article} 

\usepackage{amsmath,amssymb,amsthm,amsfonts} 
\usepackage[mathscr]{eucal}

\usepackage{hyperref}
\hypersetup{colorlinks, linkcolor=blue, citecolor=blue, urlcolor=blue}

\allowdisplaybreaks

\makeatletter
\long\def\@makecaption#1#2{%
  \vskip\abovecaptionskip\footnotesize
  \sbox\@tempboxa{#1. #2}%
  \ifdim \wd\@tempboxa >\hsize
    #1. #2\par
  \else
    \global \@minipagefalse
    \hb@xt@\hsize{\hfil\box\@tempboxa\hfil}%
  \fi
  \vskip\belowcaptionskip}
\makeatother

\newcommand{\p}{\partial}
\newcommand{\ord}{\mathop{\rm ord}\nolimits}

\newcommand{\lsemioplus}{\mathbin{\mbox{$\lefteqn{\hspace{.77ex}\rule{.4pt}{1.2ex}}{\in}$}}}

\newtheorem{theorem}{Theorem}

\newtheorem{corollary}[theorem]{Corollary}

{\theoremstyle{definition}

}

\newcommand{\todo}[1][\null]{\ensuremath{\clubsuit}}

\newcommand{\noprint}[1]{}

\begin{document}

\par\noindent {\LARGE\bf
Generalized symmetries of Burgers equation
\par}

\vspace{5.5mm}\par\noindent{\large
\large Dmytro R. Popovych$^{\dag\ddag}$,  Alex Bihlo$^\dag$ and Roman O. Popovych$^{\S\ddag}$
}
	
\vspace{5.5mm}\par\noindent{\it\small
$^\dag$Department of Mathematics and Statistics, Memorial University of Newfoundland,\\
$\phantom{^\dag}$\,St.\ John's (NL) A1C 5S7, Canada
\par}

\vspace{2mm}\par\noindent{\it\small
$^\ddag$Institute of Mathematics of NAS of Ukraine, 3 Tereshchenkivska Str., 01024 Kyiv, Ukraine
\par}

\vspace{2mm}\par\noindent{\it\small
$^\S$\,Mathematical Institute, Silesian University in Opava, Na Rybn\'\i{}\v{c}ku 1, 746 01 Opava, Czech Republic
\par}

\vspace{5mm}\par\noindent
E-mails:
dpopovych@mun.ca, abihlo@mun.ca, rop@imath.kiev.ua
	
\vspace{6mm}\par\noindent\hspace*{10mm}\parbox{140mm}{\small
Despite the number of relevant considerations in the literature, 
the algebra of generalized symmetries of the Burgers equation has not been exhaustively described.
We fill this gap, presenting a basis of this algebra in an explicit form 
and proving that the two well-known recursion operators of the Burgers equation 
and two seed generalized symmetries, which are evolution forms of its Lie symmetries, 
suffice to generate this algebra. 
The core of the proof is essentially simplified by using the original technique 
of choosing special coordinates in the associated jet space.   
}\par\vspace{4mm}
	
\noprint{


35-XX Partial differential equations
  35Kxx Parabolic equations and parabolic systems {For global analysis, analysis on manifolds, see 58J35}
    35K05 Heat equation
    35K10 Second-order parabolic equations
    35K70 Ultraparabolic equations, pseudoparabolic equations, etc.
  35Qxx	Partial differential equations of mathematical physics and other areas of application [See also 35J05, 35J10, 35K05, 35L05]
    35Q84 Fokker-Planck equations {For fluid mechanics, see 76X05, 76W05; for statistical mechanics, see 82C31}
  35Axx General topics
    35A30 Geometric theory, characteristics, transformations [See also 58J70, 58J72]
  35Bxx Qualitative properties of solutions
    35B06 Symmetries, invariants, etc.
  35Cxx Representations of solutions
    35C05 Solutions in closed form
    35C06 Self-similar solutions
    35C07 Traveling wave solutions
}

\section{Introduction}

Finding generalized symmetries of systems of differential equations is an important common problem 
in the field of symmetry analysis of differential equations and integrability theory, 
which has been intensively studied for several decades \cite{boch1999A,ibra1985A,kras1986A,olve1993A}. 
Nevertheless, there are not so many examples of correctly computing
the complete algebras of generalized symmetries of systems of differential equations in the literature, 
see, e.g., the review in~\cite{opan2020e}. 
Such algebras are not known for most famous models of mathematical physics 
that arise in real-world applications.
This is the case for the Burgers equation as well, 
despite the fact that in the context of generalized symmetries, 
this equation already appeared as an illustrative example in the pioneering paper~\cite{olve1977a} on recursion operators
and was later studied in each monograph concerning generalized symmetries, 
see 
\cite[Section~18.1]{ibra1985A},
\cite[Section~8.2]{kras1986A}, 
\cite[Section~4.4.2]{boch1999A}, 
\cite[Example 5.32]{olve1993A}.

The purpose of the present paper is to exhaustively describe generalized symmetries 
of the Burgers equation and it potential counterpart. 
This description might be derived via accurately analyzing and merging known results, 
which are two recursion operators for the Burgers equation, 
see~\cite[Eq.~(18.23)]{ibra1985A} or \cite[Eq.~(5.47)]{olve1993A},
and the number of independent generalized symmetries up to an arbitrary fixed order 
found in \cite[Section~8.2.6]{kras1986A}.
Nevertheless, we prefer to make a closed and simple proof from scratch, 
based on the linearization of the Burgers equation to the (1+1)-dimensional (linear) heat equation 
\begin{gather}\label{eq:LinHeat}
\mathcal L_1\colon\ u_t=u_{xx}
\end{gather}
by the Hopf--Cole transformation and 
the exhaustive description of generalized symmetries of the latter equation in \cite[Section~6]{kova2023b}.
The substitution $u={\rm e}^w$ reduces the equation~\eqref{eq:LinHeat}
to the potential Burgers equation
\begin{gather}\label{eq:PotBurgers}
\mathcal L_2\colon\ w_t=w_{xx}+w_x^{\,2},
\end{gather}
Differentiating the equation~\eqref{eq:PotBurgers} with respect to~$x$ and substituting $-2w_x=v$, 
one obtains the Burgers equation  
\begin{gather}\label{eq:Burgers}
\mathcal L_3\colon\ v_t+vv_x=v_{xx},
\end{gather}
which is thus related to the equation~\eqref{eq:LinHeat} by the substitution $v=-2u_x/u$, 
called the Hopf--Cole transformation.

For convenience, we denote the solution set of the equation~$\mathcal L_i$, $i\in\{1,2,3\}$, 
and the associated manifold in the underlying infinite-order jet space by the same symbol~$\mathcal L_i$. 
We also use the other, temporary notation~$z^i$ for the corresponding dependent variable, $z^1:=u$, $z^2:=w$ and~$z^3=w$. 
Since $\mathcal L_i$ is an evolution equation, 
we can choose~$z^i$ and the derivatives of~$z^i$ with respect to~$x$
as the parametric derivatives of~$\mathcal L_i$, 
and the derivatives involving differentiations with respect to~$t$ as the principal derivatives of~$\mathcal L_i$. 
Then all differential functions of~$z^i$ related to~$\mathcal L_i$ can be replaced by their reduced counterparts,
which depends on~$t$, $x$ and parametric derivatives only. 
Instead of the total derivative operators with respect to~$t$ and~$x$, 
we use their restrictions to~$\mathcal L_i$ in the above parameterization, 
\[
\mathrm D_x:=\p_x+\sum_{k=0}^\infty z^i_{k+1}\p_{z^i_k},\quad 
\mathrm D_t:=\p_t+\sum_{k=0}^\infty \big(\mathrm D_x^kL^i[z^i]\big)\p_{z^i_k},
\]   
where $L^1[u]:=u_{xx}$, $L^2[w]:=w_{xx}+w_x^{\,2}$, $L^3[v]:=v_{xx}-vv_x$, 
$z^i_0:=z^i$, and 
the jet variable~$z^i_k$ is identified with the derivative $\p^k z^i/\p x^k$, $k\in\mathbb N$. 
We can also naturally identify the quotient algebra of generalized symmetries of this equation
with respect to the equivalence of generalized symmetries
with the algebra of canonical representatives of equivalence classes, 
which are in the evolution form and have reduced characteristics. 
For brevity, we call the latter algebra the algebra of generalized symmetries of this equation. 

The structure of the paper is as follows.
After briefly presenting the required results on generalized symmetries of the heat equation~\eqref{eq:LinHeat} 
from \cite[Section~6]{kova2023b} in Section~\ref{sec:LinHeatGenSyms}, 
we translate these results to those for the potential Burgers equation~\eqref{eq:PotBurgers} 
via ``pushing forward'' by the transformation $u={\rm e}^w$ in Section~\ref{sec:PotBurgersGenSyms}. 
The substitution $-2w_x=v$ induces a homomorphism of the essential algebra of generalized symmetries 
of the potential Burgers equation~\eqref{eq:PotBurgers} 
to the algebra of generalized symmetries of the Burgers equation~\eqref{eq:Burgers}. 
We prove in Section~\ref{sec:BurgersGenSyms} that this homomorphism is in fact a surjection, 
which gives the complete description of the latter algebra, 
its structure and recursion operators generating it. 
The computations required for the proof are essentially simplified 
by using the original technique of choosing special coordinates in the associated jet space 
over $\mathbb R^2_{t,x}\times\mathbb R_v$.

\section{Generalized symmetries of heat equation}\label{sec:LinHeatGenSyms}

\begin{theorem}\label{thm:LinHeatAlgOfGenSyms}
The algebra of generalized symmetries of the (1+1)-dimensional linear heat equation~\eqref{eq:LinHeat} is
$\Sigma_1=\Lambda_1\lsemioplus\Sigma_1^{-\infty}$, where
\begin{gather*}
\Lambda_1:=\big\langle\mathfrak Q^{kl},\,k,l\in\mathbb N_0\big\rangle,
\quad
\Sigma_1^{-\infty}:=\big\{\mathfrak Z(h)\big\}
\end{gather*}
with $\mathfrak Q^{kl}:=(\mathrm G^k\mathrm P^lu)\p_u$, 
$\mathrm P:=\mathrm D_x$,
$\mathrm G:=t\mathrm D_x+\frac12x$,
$\mathfrak Z(h):=h(t,x)\p_u$,
and the parameter function~$h$ runs through the solution set of~\eqref{eq:LinHeat}.
\end{theorem}

Elements of~$\Sigma_1^{-\infty}$ are considered trivial generalized symmetries of~\eqref{eq:LinHeat}
since in fact these are Lie symmetries of~\eqref{eq:LinHeat} 
that are associated with the linear superposition of solutions of~\eqref{eq:LinHeat}. 
By analogy with the essential Lie invariance algebra 
$\mathfrak g^{\rm ess}=\langle\mathcal P^t,\mathcal D,\mathcal K,\mathcal G^x,\mathcal P^x,\mathcal I\rangle$ 
of~\eqref{eq:LinHeat}, where 
\begin{gather*}
\mathcal P^t =\p_t, \quad
\mathcal D   =2t\p_t+x\p_x-\tfrac12u\p_u,\quad
\mathcal K   =t^2\p_t+tx\p_x-\tfrac14(x^2+2t)u\p_u,\\
\mathcal G^x =t\p_x-\tfrac12xu\p_u,\quad
\mathcal P^x =\p_x,\quad
\mathcal I   =u\p_u,
\end{gather*}
see~\cite[Section~2]{kova2023b}, we can call 
the complement subalgebra~$\Lambda_1$ of~$\Sigma_1^{-\infty}$ in~$\Sigma_1$, 
which is constituted by the linear generalized symmetries of the equation~\eqref{eq:LinHeat},
the essential algebra of generalized symmetries of this equation.
The algebra~$\Lambda_1$
is generated by the two recursion operators~$\mathrm D_x$ and~$\mathrm G$ from the simplest linear generalized symmetry~$u\p_u$,
and both the recursion operators and the seed symmetry are related to Lie symmetries.
In particular, on the solution set of the equation~\eqref{eq:LinHeat},
the generalized symmetries associated with the basis elements
$\mathcal P^t$, $\mathcal D$, $\mathcal K$, $\mathcal G^x$, $\mathcal P^x$ and $\mathcal I$
of~$\mathfrak g^{\rm ess}$ are, up to sign, $\mathfrak Q^{02}$, $2\mathfrak Q^{11}+\frac12\mathfrak Q^{00}$,
$\mathfrak Q^{20}$, $\mathfrak Q^{10}$, $\mathfrak Q^{01}$, $\mathfrak Q^{00}$, respectively,
see \cite[Example~5.21]{olve1993A}.

Since $\mathrm P\mathrm G=\mathrm G\mathrm P+\frac12$,
the commutation relations between the generalized vector fields spanning the algebra~$\Sigma_1$ are the following:
\begin{gather}\label{eq:LinHeatGenSymsCommRels1}
[\mathfrak Q^{kl},\mathfrak Q^{k'l'}]
=\sum_{i=0}^{\min(k,l')}\frac{i!}{2^i}\binom ki\binom{l'}i\mathfrak Q^{k+k'-i,\,l+l'-i}
-\sum_{i=0}^{\min(k',l)}\frac{i!}{2^i}\binom{k'}i\binom li\mathfrak Q^{k+k'-i,\,l+l'-i},
\\[1ex]\label{eq:LinHeatGenSymsCommRels2}
[\mathfrak Z(h),\mathfrak Q^{kl}]=\mathfrak Z(\mathrm G^k\mathrm D_x^lh), \quad
[\mathfrak Z(h^1),\mathfrak Z(h^2)]=0.
\end{gather}

\section{Generalized symmetries of potential Burgers equation}\label{sec:PotBurgersGenSyms}

Pulling back the elements of the algebra~$\Sigma_1$ by the transformation~$u={\rm e}^w$, 
we obtain the algebra~$\Sigma_2$ of generalized symmetries of the potential Burgers equation~\eqref{eq:PotBurgers}.

\begin{theorem}\label{thm:PotBurgersAlgOfGenSyms}
The algebra of generalized symmetries of the potential Burgers equation~\eqref{eq:PotBurgers} is
$\Sigma_2=\Lambda_2\lsemioplus\Sigma_2^{-\infty}$, where
\begin{gather*}
\Lambda_2:=\big\langle\tilde{\mathfrak Q}^{kl},\,k,l\in\mathbb N_0\big\rangle,
\quad
\Sigma_2^{-\infty}:=\big\{\tilde{\mathfrak Z}(h)\big\}
\end{gather*}
with 
$\tilde{\mathfrak Q}^{kl}:=(\tilde{\mathrm G}^k\tilde{\mathrm P}^l1)\p_w$, 
$\tilde{\mathrm P}:=\mathrm D_x+w_x$,
$\tilde{\mathrm G}:=t\tilde{\mathrm P}+\frac12x$,
$\tilde{\mathfrak Z}(h):=h(t,x){\rm e}^{-w}\p_w$,
and the parameter function~$h$ runs through the solution set of~\eqref{eq:LinHeat}.
\end{theorem}

It is clear that the algebra~$\Sigma_2$ is isomorphic to the algebra~$\Sigma_1$, 
and thus the generalized vector fields~$\tilde{\mathfrak Q}^{kl}$ and~$\tilde{\mathfrak Z}(h)$ 
spanning the algebra~$\Sigma_2$ satisfy the same commutation relations 
as~\eqref{eq:LinHeatGenSymsCommRels1},~\eqref{eq:LinHeatGenSymsCommRels2},
\begin{gather}\label{eq:PotBurgersGenSymsCommRels1}
[\tilde{\mathfrak Q}^{kl},\tilde{\mathfrak Q}^{k'l'}]
=\sum_{i=0}^{\min(k,l')}\frac{i!}{2^i}\binom ki\binom{l'}i\tilde{\mathfrak Q}^{k+k'-i,\,l+l'-i}
-\sum_{i=0}^{\min(k',l)}\frac{i!}{2^i}\binom{k'}i\binom li\tilde{\mathfrak Q}^{k+k'-i,\,l+l'-i},
\\[1ex]\nonumber
[\tilde{\mathfrak Z}(h),\tilde{\mathfrak Q}^{kl}]=\tilde{\mathfrak Z}(\mathrm G^k\mathrm D_x^lh), \quad
[\tilde{\mathfrak Z}(h^1),\tilde{\mathfrak Z}(h^2)]=0.
\end{gather}

As the counterparts of~$\Lambda_1$ and~$\Sigma_1^{-\infty}$, 
the subalgebra~$\Lambda_2$ and the ideal~$\Sigma_2^{-\infty}$ of~$\Sigma_2$ are called 
the essential and the trivial algebras of generalized symmetries of~\eqref{eq:PotBurgers}, respectively.

Theorem~\ref{thm:LinHeatAlgOfGenSyms} implies that the potential Burgers equation~\eqref{eq:PotBurgers} 
admits two independent recursion operators, $\tilde{\mathrm P}$ and $\tilde{\mathrm G}$, 
cf.\ \cite[Eq.~(18.17)]{ibra1985A} and \cite[Example 5.30]{olve1993A}, 
and this set of recursion operators is complete in the sense 
that it suffices to generate the essential part~$\Lambda_2$ of the algebra~$\Sigma_2$.  
The seed symmetry is $\tilde{\mathfrak Q}^{00}=\p_w$, 
which is the evolution form of a Lie symmetry of~\eqref{eq:PotBurgers}.

\section{Generalized symmetries of Burgers equation}\label{sec:BurgersGenSyms}

In view of the relation between the potential Burgers equation~\eqref{eq:PotBurgers}
and the Burgers equation~\eqref{eq:Burgers} via the differential substitution $-2w_x=v$, 
we may try to map each generalized symmetry vector field~$Q$ of~\eqref{eq:PotBurgers} to a one of~\eqref{eq:Burgers}
in three steps, 
(i) prolong~$Q$ to the derivative~$w_x$,
(ii) if possible, push forward the prolonged vector field $\mathop{\rm pr}\nolimits_1Q$ 
by the natural projection $\pi\colon\mathbb R^4_{t,x,w,w_x}\to\mathbb R^3_{t,x,w_x}$, and
(iii) replace $w_x$ by $-\frac12v$,
cf.\ \cite[Example 5.32]{olve1993A}. 

This procedure leads to a well-defined homomorphism~$\varphi$ of~$\Lambda_2$ 
in the algebra~$\Sigma_3$ of generalized symmetries of the Burgers equation~\eqref{eq:Burgers}, 
\[
\tilde{\mathfrak Q}^{kl}:=(\tilde{\mathrm G}^k\tilde{\mathrm P}^l1)\p_w
\to
\mathop{\rm pr}\nolimits_1\tilde{\mathfrak Q}^{kl}
=(\tilde{\mathrm G}^k\tilde{\mathrm P}^l1)\p_w+(\mathrm D_x\tilde{\mathrm G}^k\tilde{\mathrm P}^l1)\p_{w_x}
\to
\pi_*\mathop{\rm pr}\nolimits_1\tilde{\mathfrak Q}^{kl}=-2\hat{\mathfrak Q}^{kl},
\] 
where $k,l\in\mathbb N_0$, 
$\hat{\mathfrak Q}^{kl}:=\big(\mathrm D_x\hat{\mathrm G}^k\hat{\mathrm P}^l1\big)\p_v$,
$\hat{\mathrm P}:=\mathrm D_x-\tfrac12v$ and 
$\hat{\mathrm G}:=t\hat{\mathrm P}+\tfrac12x=t(\mathrm D_x-\tfrac12v)+\tfrac12x$. 
The kernel of this homomorphism is spanned by $\tilde{\mathfrak Q}^{00}$.
Denote its image by~$\Lambda_3$, 
\[
\Lambda_3=\varphi(\Lambda_2)
=\big\langle\hat{\mathfrak Q}^{kl},\,(k,l)\in\mathbb N_0^{\,2}\setminus\{(0,0)\}\big\rangle
\subseteq\Sigma_3.
\] 
The linear independence of the generalized vector fields~$\hat{\mathfrak Q}^{kl}$, $(k,l)\in\mathbb N_0^{\,2}\setminus\{(0,0)\}$,
is obvious, and the fact that they are indeed generalized symmetries of~\eqref{eq:Burgers} 
can be checked independently using the following commutation relations:
\begin{gather*}
[\mathrm D_t+v\mathrm D_x-\mathrm D_x^2,\hat{\mathrm P}]=
-\tfrac12(v_t+vv_x-v_{xx})+v_x\mathrm D_x-v_x\mathrm D_x=0,
\\
[\mathrm D_t+v\mathrm D_x-\mathrm D_x^2,\hat{\mathrm G}]=
\mathrm D_x-\tfrac12v+t[\mathrm D_t+v\mathrm D_x-\mathrm D_x^2,\hat{\mathrm P}]+\tfrac12v-\mathrm D_x=0,
\\
(\mathrm D_t+v\mathrm D_x+v_x-\mathrm D_x^2)\mathrm D_x\hat{\mathrm G}^k\hat{\mathrm P}^l1
=
\mathrm D_x(\mathrm D_t+v\mathrm D_x-\mathrm D_x^2)\hat{\mathrm G}^k\hat{\mathrm P}^l1=0.
\end{gather*}
At the same time, this procedure cannot be applied to
generalized vector field with nonvanishing value of the parameter function~$h$,   
since the prolongation of~$\tilde{\mathfrak Z}(h)$ to the derivative~$w_x$,  
$\mathop{\rm pr}\nolimits_1\tilde{\mathfrak Z}(h)=h(t,x){\rm e}^{-w}\p_w+(h_x+hw_x){\rm e}^{-w}\p_{w_x}$,
is not projectable to the space $\mathbb R^3_{t,x,w_x}$. 
In other words, the generalized symmetries~$\tilde{\mathfrak Z}(h)$ 
of the potential Burgers equation~\eqref{eq:PotBurgers} 
have no counterparts among generalized symmetries of the Burgers equation~\eqref{eq:Burgers} 
and correspond to purely potential symmetries of the latter equation.

\begin{theorem}\label{thm:BurgersAlgOfGenSyms}
The algebra of generalized symmetries of the Burgers equation~\eqref{eq:Burgers} is
\begin{gather*}
\Sigma_3:=\big\langle\hat{\mathfrak Q}^{kl},\,(k,l)\in\mathbb N_0^{\,2}\setminus\{(0,0)\}\big\rangle
\quad\mbox{with}\quad
\hat{\mathfrak Q}^{kl}:=\big(\mathrm D_x\hat{\mathrm G}^k\hat{\mathrm P}^l1\big)\p_v,
\end{gather*}
where 
$\hat{\mathrm P}:=\mathrm D_x-\tfrac12v$ and 
$\hat{\mathrm G}:=
t\mathrm D_x+\tfrac12(x-vt)$.
\end{theorem}

\begin{proof}
The statement of the theorem is equivalent to the equality $\Sigma_3=\Lambda_3$.
The subspace
$\Sigma_3^n:=\big\{\eta[v]\p_v\in\Sigma\mid\ord\eta[v]\leqslant n\big\}$, $n\in\mathbb N_0\cup\{-\infty\}$,
of~$\Sigma_3$ is interpreted as the space of generalized symmetries of orders less than or equal to~$n$. 
Recall that the order~$\ord F[v]$ of a differential function~$F[v]$ of~$v$ 
is the highest order of derivatives of~$v$ involved in~$F[v]$
if there are such derivatives, and $\ord F[v]=-\infty$ otherwise.
If $Q=\eta[v]\p_v$, then $\ord Q:=\ord\eta[v]$.
The subspace family $\{\Sigma_3^n\mid n\in\mathbb N_0\cup\{-\infty\}\}$ filters the algebra~$\Sigma_3$.
Denote 
\smash{$\Sigma_3^{[n]}:=\Sigma_3^n/\Sigma_3^{n-1}$}, $n\in\mathbb N$,
\smash{$\Sigma_3^{[0]}:=\Sigma_3^0/\Sigma_3^{-\infty}$} and 
\smash{$\Sigma_3^{[-\infty]}:=\Sigma_3^{-\infty}$}.
The space~\smash{$\Sigma_3^{[n]}$} is naturally identified with the space of canonical representatives of cosets of~$\Sigma_3^{n-1}$
and thus assumed as the space of $n$th order generalized symmetries of the equation~\eqref{eq:Burgers},
$n\in\mathbb N_0\cup\{-\infty\}$.
Since $\ord\hat{\mathfrak Q}^{kl}=k+l$, we have 
$\Lambda_3^n:=\big\langle\hat{\mathfrak Q}^{kl},\,k+l\leqslant n\big\rangle\subseteq\Sigma_3^n$, 
$n\in\mathbb N$.
We set $\Lambda_3^{-\infty}$, $\Lambda_3^0$, \smash{$\Lambda_3^{[-\infty]}$} and \smash{$\Lambda_3^{[0]}$}
to be equal to the null subspace of~$\Lambda_3$, and 
\smash{$\Lambda_3^{[n]}:=\Lambda_3^n/\Lambda_3^{n-1}\simeq\big\langle\hat{\mathfrak Q}^{kl},\,k+l=n\big\rangle$}, 
$n\in\mathbb N$.
Since $\Lambda_3\subseteq\Sigma_3$, to prove $\Sigma_3=\Lambda_3$
it suffices to check that $\dim\Sigma_3^n=\dim\Lambda_3^n$ for any $n\in\mathbb N_0\cup\{-\infty\}$, 
which is equivalent, in view of the inclusion $\Lambda_3^n\subseteq\Sigma_3^n$, 
to the equality $\dim\Sigma_3^n=\dim\Lambda_3^n$.
We will prove that \smash{$\dim\Sigma_3^{[n]}=\dim\Lambda_3^{[n]}$}.

Instead of $(t,x,v_k,k\in\mathbb N)$, 
we use $(t,x,\zeta^k,k\in\mathbb N)$ with $\zeta^k:=\hat{\mathrm P}^{k+1}1$, $k\in\mathbb N$,
as coordinates on~$\mathcal L_3$.
We have  
$\zeta^0=-\tfrac12v$, $\zeta^1=-\tfrac12v_x+\tfrac14v^2$, i.e., 
$v=-2\zeta^0$, $v_x=-2\big(\zeta^1-(\zeta^0)^2\big)$,
$\ord\zeta^k=k$,
\begin{gather*}
(\mathrm D_t+v\mathrm D_x-\mathrm D_x^2)\zeta^k=0,
\quad
\mathrm D_x\zeta^k=\left(\hat{\mathrm P}+\frac v2\right)\hat{\mathrm P}^{k+1}1=\zeta^{k+1}-\zeta^0\zeta^k.
\end{gather*}
In the new coordinates, the generalized invariance condition
$\mathrm D_t\eta+v\mathrm D_x\eta+v_x\eta-\mathrm D_x^2\eta=0$ 
for an $n$th-order (reduced) generalized symmetry $Q=\eta(t,x,\zeta^0,\zeta^1,\dots,\zeta^n)\p_v$ 
of the Burgers equation~\eqref{eq:Burgers} reduces to the equation  
\begin{gather*}
\noprint{
\eta_t+\underline{\underline{\eta_{\zeta^l}\mathrm D_t\zeta^l}}
+v(\eta_x+\underline{\underline{\eta_{\zeta^l}\mathrm D_x\zeta^l}})
+v_x\eta
-\big(\eta_{xx}+2\eta_{x\zeta^l}\mathrm D_x\zeta^l+\eta_{\zeta^l\zeta^{l'}}(\mathrm D_x\zeta^l)\mathrm D_x\zeta^{l'}
+\underline{\underline{\eta_{\zeta^l}\mathrm D_x^2\zeta^l}}\big)
\\=
}
\eta_t-2\zeta^0\eta_x-\eta_{xx}-2\eta\big(\zeta^1-(\zeta^0)^2\big)
-\big(2\eta_{x\zeta^l}+\eta_{\zeta^l\zeta^{l'}}(\zeta^{l'+1}-\zeta^0\zeta^{l'})\big)(\zeta^{l+1}-\zeta^0\zeta^l)=0.
\end{gather*}
Here and in what follows the indices~$l$ and~$l'$ run from~0 to~$n$, 
and we assume summation with respect to repeated indices. 
If $\ord\eta\leqslant0$, then the splitting of this equation with respect to~$\zeta^1$ implies that $\eta=0$. 
This is why we can assume that $\ord\eta=:n>0$.
For each $j\in\{2,\dots,n+1\}$, 
we differentiate this equation with respect to~$\zeta^j$,
denote the result of differentiation by $\Delta_j$ 
and write it using the notation $\eta^k:=\eta_{\zeta^k}$, $k=1,2,\dots$, 
\begin{gather*}
\Delta_j, \quad j=2,\dots,n+1 \colon\\ 
\eta^j_t-2\zeta^0\eta^j_x-\eta^j_{xx}-2\eta^j\big(\zeta^1-(\zeta^0)^2\big)
-\big(2\eta^j_{x\zeta^l}+\eta^j_{\zeta^l\zeta^{l'}}(\zeta^{l'+1}-\zeta^0\zeta^{l'})\big)(\zeta^{l+1}-\zeta^0\zeta^l)
\\\qquad\qquad{}
-2\mathrm D_x\eta^{j-1}+2\zeta^0\mathrm D_x\eta^j
=0.
\end{gather*}

We prove by induction downward with respect to~$i$ starting from $i=n$ to $i=1$
that $\eta^i=\alpha^i(t,x)\zeta^0+\beta^i(t,x)$, 
where the coefficients~$\alpha^i$ and~$\beta^i$ satisfy the system 
\begin{gather}\label{eq:ForAlphaBeta}
\alpha^i=-(\beta^{i+1}+\alpha^{i+1}_x), \quad
2\alpha^i_x=\alpha^{i+1}_t-\alpha^{i+1}_{xx}, \quad
2\beta^i_x=\beta^{i+1}_t-\beta^{i+1}_{xx}, \quad i=1,\dots,n,
\end{gather}
with  and $\beta^{n+1}:=0$. 
Since $\eta^{n+1}=0$,
the equation~$\Delta_{n+1}$ reduces to $\mathrm D_x\eta^n=0$, 
i.e., $\alpha^n=0$ and $\eta^n=\beta^n$ is a function of~$t$. 
It is obvious that the equations~\eqref{eq:ForAlphaBeta} with $i=n$ are satisfied.  
This gives the induction base $i=n$. 
Suppose that the claim holds for $i=k+1$ for a fixed $k\in\{1,\dots,n-1\}$ and prove it for $i=k$. 
Under the supposition, the equation~$\Delta_{k+1}$ can be written as 
\begin{gather*}
\begin{split}
2\mathrm D_x\eta^k&{}=
\noprint{
\eta^{k+1}_t-2\zeta^0\eta^{k+1}_x-\eta^{k+1}_{xx}-2\big(\eta^{k+1}+\eta^{k+1}_{x\zeta^0}\big)\big(\zeta^1-(\zeta^0)^2\big)
+2\zeta^0\mathrm D_x\eta^{k+1}
\\&{}=
\zeta^0\big(\alpha^{k+1}_t-\underline{\underline{2\zeta^0\alpha^{k+1}_x}}-\alpha^{k+1}_{xx}-
\underline{2\alpha^{k+1}\big(\zeta^1-(\zeta^0)^2\big)}\big)
\\&\quad{}
+\beta^{k+1}_t-\underline{\underline{2\zeta^0\beta^{k+1}_x}}-\beta^{k+1}_{xx}
-2(\beta^{k+1}+\alpha^{k+1}_x)\big(\zeta^1-(\zeta^0)^2\big)
\\&\quad{}
+2\zeta^0\big(\underline{\underline{\alpha^{k+1}_x\zeta^0+\beta^{k+1}_x}}+\underline{\alpha^{k+1}\big(\zeta^1-(\zeta^0)^2\big)} \big)
\\&{}=
}
\zeta^0\big(\alpha^{k+1}_t-\alpha^{k+1}_{xx}\big)+\beta^{k+1}_t-\beta^{k+1}_{xx}
-2(\beta^{k+1}+\alpha^{k+1}_x)\big(\zeta^1-(\zeta^0)^2\big).
\end{split}
\end{gather*}
It implies that the function~$\eta^k$ depends at most on $(t,x,\zeta^0)$ 
and \smash{$\eta^k_{\zeta^0}=-(\beta^{k+1}+\alpha^{k+1}_x)$}. 
Hence $\eta^k=-(\beta^{k+1}+\alpha^{k+1}_x)\zeta^0+\beta^k$ 
for $\alpha^k:=-(\beta^{k+1}+\alpha^{k+1}_x)$ and some function~$\beta^k$ of~$(t,x)$,
and the equation~$\Delta_{k+1}$ reduces to the equation~\eqref{eq:ForAlphaBeta} with $i=k$.

Therefore, we derive the following representation for~$\eta$:
\[
\eta=\sum_{i=1}^n(\alpha^i\zeta^0+\beta^i)\zeta^i+\rho,
\]
where the coefficients $\alpha^i=\alpha^i(t,x)$ and $\beta^i=\beta^i(t,x)$ satisfy the system~\eqref{eq:ForAlphaBeta}
and $\rho$ is a function at most of~$(t,x,\zeta^0)$.
The system~\eqref{eq:ForAlphaBeta} implies 
that the coefficients~$\alpha^i$ and~$\beta^i$, $i=1,\dots,n$, 
are polynomials in~$x$ with coefficients depending at most on~$t$ 
and $\deg_x\alpha^i\leqslant n-i-1$, $\deg_x\beta^i\leqslant n-i$. 
Moreover, for any $i\in\{1,\dots,n-1\}$ 
the coefficient of $x^{n-i}$ in~$\beta^i$ is equal to $\p_t^{n-i}\beta^n/(2^{n-i}(n-i)!)$.
Under the derived representation for~$\eta$, 
the generalized invariance condition reduces to 
\begin{gather}\label{eq:BurgersEqGenSymCondReduced}
\begin{split}&
\big(\alpha^1_t-\alpha^1_{xx}\big)\zeta^0\zeta^1+(\beta^1_t-\beta^1_{xx})\zeta^1
-2(\beta^1+\alpha^1_x)\zeta^1\big(\zeta^1-(\zeta^0)^2\big)
\\&\qquad{}
+\rho_t-2\zeta^0\rho_x-\rho_{xx}-2(\rho+\rho_{x\zeta^0})\big(\zeta^1-(\zeta^0)^2\big)
-\rho_{\zeta^0\zeta^0}\big(\zeta^1-(\zeta^0)^2\big)^2
=0.
\end{split}
\end{gather}
Collecting the coefficients of~$(\zeta^1)^2$ in~\eqref{eq:BurgersEqGenSymCondReduced}, 
we derive the simple equation $\rho_{\zeta^0\zeta^0}=-2(\beta^1+\alpha^1_x)$, 
which integrates to 
$\rho=-(\beta^1+\alpha^1_x)(\zeta^0)^2+\gamma^1\zeta^0+\gamma^0$ 
for some functions~$\gamma^0$ and~$\gamma^1$ of~$(t,x)$. 
We substitute the derived expression for~$\rho$ into~\eqref{eq:BurgersEqGenSymCondReduced} 
and continue, splitting the result simultaneously with respect to~$\zeta^1$ and~$\zeta^0$  
and neglecting those obtained equations that are differential consequences of the others. 
This leads to the system 
\begin{gather*}
\alpha^1_t-\alpha^1_{xx}+2(\beta^1_x+\alpha^1_{xx})=0,\quad
\gamma^1=\beta^1_x+\alpha^1_{xx},\\
2\gamma^0=\beta^1_t-\beta^1_{xx}+\alpha^1_{tx}-\alpha^1_{xxx},\quad
\gamma^0_t-\gamma^0_{xx}=0.
\end{gather*}

As a result, $\p_t^{n+1}\beta^n=0$.
This means that $\dim\Sigma_3^{[n]}\leqslant n+1$. 
Since also $\dim\Sigma_3^{[n]}\geqslant\dim\Lambda_3^{[n]}=n+1$, 
we obtain $\dim\Sigma_3^{[n]}=n+1=\dim\Lambda_3^{[n]}$. 
Therefore, $\dim\Sigma_3=\dim\Lambda_3$.
\end{proof}

\begin{corollary}\label{cor:BurgersGenSymsHomomorphism}
The homomorphism $\varphi\colon\Lambda_2\to\Sigma_3$ 
of the algebra~$\Lambda_2$ of essential generalized symmetries of the potential Burgers equation~\eqref{eq:PotBurgers} 
to the entire algebra~$\Lambda_2$ of generalized symmetries of the Burgers equation~\eqref{eq:Burgers}, 
which is induced by the differential substitution $-2w_x=v$, is a surjection.
\end{corollary}

\begin{corollary}\label{cor:BurgersGenSymsAndImDx}
The characteristic of any generalized symmetry vector field of the Burgers equation~\eqref{eq:Burgers} 
belongs to the image of~$\mathrm D_x$.
\end{corollary}

\begin{corollary}\label{cor:BurgersRecursionOps}
The nonlocal operators 
\begin{gather*}
\mathrm R_1:=\mathrm D_x\hat{\mathrm P}\mathrm D_x^{-1}=\mathrm D_x-\tfrac12v-\tfrac12v_x\mathrm D_x^{-1},\\
\mathrm R_2:=\mathrm D_x\hat{\mathrm G}\mathrm D_x^{-1}=t\mathrm D_x+\tfrac12(x-vt)+\tfrac12(1-tv_x)\mathrm D_x^{-1}
\end{gather*}
are recursion operators of the Burgers equation~\eqref{eq:Burgers} 
that suffice for generating a basis of the algebra~$\Sigma_3$ of generalized symmetries of this equation.
A sufficient set of seed generalized symmetries consists of two symmetries, 
$\hat{\mathfrak Q}^{01}=-\tfrac12v_x\p_v$ and $\hat{\mathfrak Q}^{10}=-\tfrac12(x-tv_x)\p_v$. 
\end{corollary}

The seed generalized symmetries $\hat{\mathfrak Q}^{01}$ and $\hat{\mathfrak Q}^{10}$ 
are the evolution forms of the Lie symmetries $-\tfrac12\mathcal P^x$ and~$-\tfrac12\mathcal G^x$, respectively.
There is a relation between them in terms of the recursion operators~$\mathrm R_1$ and~$\mathrm R_2$, 
$\mathrm R_2\hat{\mathfrak Q}^{01}=\mathrm R_1\hat{\mathfrak Q}^{10}$, 
and the commutator of the recursion operators is
\[
[\mathrm R_1,\mathrm R_2]=\mathrm D_x[\hat{\mathrm P},\hat{\mathrm G}]\mathrm D_x^{-1}=\tfrac12.
\]
This is why there are unique ways for generating the symmetries $\hat{\mathfrak Q}^{0l}$ and $\hat{\mathfrak Q}^{k0}$ 
with $k,l\in\mathbb N$, 
$\hat{\mathfrak Q}^{0l}=\mathrm R_1^{l-1}\hat{\mathfrak Q}^{01}$ and 
$\hat{\mathfrak Q}^{k0}=\mathrm R_2^{k-1}\hat{\mathfrak Q}^{10}$,
whereas for each symmetry $\hat{\mathfrak Q}^{kl}$ with $k,l\in\mathbb N$ there are, 
even up to permuting the actions of~$\mathrm R_1$ and~$\mathrm R_2$, 
different ways to generate it, e.g.,
\[
\hat{\mathfrak Q}^{kl}
=\mathrm R_2^k\mathrm R_1^{l-1}\hat{\mathfrak Q}^{01}
=\mathrm R_2^{k-1}\mathrm R_1^l\hat{\mathfrak Q}^{10}-\tfrac12(l-1)\hat{\mathfrak Q}^{k-1,l-1}.
\]
\noprint{
\begin{gather*}
\mathrm R_2^{k-1}\mathrm R_1^{l-1}\mathrm R_1\hat{\mathfrak Q}^{10}
=\mathrm R_2^{k-1}\mathrm R_1^{l-1}\mathrm R_2\hat{\mathfrak Q}^{01}
=\mathrm R_2^{k-1}\mathrm R_1^{l-2}(\mathrm R_2\mathrm R_1+\tfrac12)\hat{\mathfrak Q}^{01}
\\\qquad{}
=\mathrm R_2^{k-1}\mathrm R_1^{l-2}\mathrm R_2\mathrm R_1\hat{\mathfrak Q}^{01}+\tfrac12\hat{\mathfrak Q}^{k-1,l-1}
=\mathrm R_2^{k-1}\mathrm R_1^{l-3}(\mathrm R_2\mathrm R_1+\tfrac12)\mathrm R_1\hat{\mathfrak Q}^{01}+\tfrac12\hat{\mathfrak Q}^{k-1,l-1}
\\\qquad{}
=\mathrm R_2^{k-1}\mathrm R_1^{l-3}\mathrm R_2\mathrm R_1^2\hat{\mathfrak Q}^{01}+2\cdot\tfrac12\hat{\mathfrak Q}^{k-1,l-1}
=\dots
=\mathrm R_2^k\mathrm R_1^{l-1}\hat{\mathfrak Q}^{01}+\tfrac12(l-1)\hat{\mathfrak Q}^{k-1,l-1}
\\\qquad{}
=\hat{\mathfrak Q}^{kl}+\tfrac12(l-1)\hat{\mathfrak Q}^{k-1,l-1}
\end{gather*}
}

Since the algebra~$\Sigma_3$ is homomorphic to the algebra~$\Lambda_2$, 
the same commutation relations~\eqref{eq:PotBurgersGenSymsCommRels1} 
imply those of the algebra~$\Sigma_3$, 
\begin{gather*}
[\hat{\mathfrak Q}^{kl},\hat{\mathfrak Q}^{k'l'}]
=\sum_{i=0}^{\min(k,l')}\frac{i!}{2^i}\binom ki\binom{l'}i\hat{\mathfrak Q}^{k+k'-i,\,l+l'-i}
-\sum_{i=0}^{\min(k',l)}\frac{i!}{2^i}\binom{k'}i\binom li\hat{\mathfrak Q}^{k+k'-i,\,l+l'-i},\\
(k,l),(k',l')\in\mathbb N_0^{\,2}\setminus\{(0,0)\}.
\end{gather*}

\subsection*{Acknowledgements}

The authors are grateful to Serhii Koval for valuable discussions.
This research was undertaken thanks to funding from the Canada Research Chairs program,
the InnovateNL LeverageR{\&}D program and the NSERC Discovery Grant program.
It was also supported in part by the Ministry of Education, Youth and Sports of the Czech Republic (M\v SMT \v CR)
under RVO funding for I\v C47813059.
ROP expresses his gratitude for the hospitality shown by the University of Vienna during his long stay at the university.




\begin{thebibliography}{10}\footnotesize\itemsep=0.ex

\bibitem{boch1999A}
Bocharov A.V., Chetverikov V.N., Duzhin S.V., Khor'kova N.G., Krasil'shchik I.S., Samokhin A.V., Torkhov~Y.N., Verbovetsky~A.M. and Vinogradov~A.M.,
{\it Symmetries and conservation laws for differential equations of mathematical physics},
American Mathematical Society, Providence, 1999.

\bibitem{ibra1985A}
Ibragimov~N.H.,
{\it Transformation groups applied to mathematical physics},
D. Reidel Publishing Co., Dordrecht, 1985.

\bibitem{kras1986A}
Krasil'shchik I.S., Lychagin V.V. and Vinogradov A.M.,
{\it Geometry of jet spaces and nonlinear partial differential equations},
Gordon and Breach Science Publishers, New York, 1986.

\bibitem{olve1977a}
Olver P.J.,
Evolution equations possessing infinitely many symmetries
{\it J.~Math. Phys.} {\bf 18} (1977), 1212--1215.

\bibitem{olve1993A}
Olver P.J.,
{\it Application of Lie groups to differential equations},
Springer, New York, 1993.

\bibitem{opan2020e}
Opanasenko S. and Popovych R.O., 
Generalized symmetries and conservation laws of (1+1)-dimensional Klein--Gordon equation, 
{\it J.~Math. Phys.} {\bf 61} (2020), 101515, arXiv:1810.12434.

\bibitem{kova2023b}
Koval S.D. and Popovych R.O.,
Point and generalized symmetries of the heat equation revisited,
{\it J.~Math. Anal. Appl.} {\bf 527} (2023), 127430, arXiv:2208.11073.



\end{thebibliography}
\end{document}